\documentclass[12pt]{article}%
\usepackage{graphicx}
\usepackage{color}
\usepackage{amsmath}
\usepackage{a4}
\usepackage{amsfonts}
\usepackage{mathtools}
\usepackage{amssymb}%
\providecommand{\U}[1]{\protect \rule{.1in}{.1in}}
\newtheorem{theorem}{Theorem}

\newtheorem{corollary}[theorem]{Corollary}

\newtheorem{lemma}[theorem]{Lemma}

\newtheorem{proposition}[theorem]{Proposition}
\newtheorem{remark}[theorem]{Remark}

\newenvironment{proof}[1][Proof]{\textbf{#1.} }{\hfill  \rule{0.5em}{0.5em}}

\begin{document}
 
\title{Convexity of asymptotic geodesics in Hilbert Geometry}
\author{Charalampos Charitos, Ioannis Papadoperakis
\and and Georgios Tsapogas\\Agricultural University of Athens}
\maketitle

\begin{abstract}
If $\Omega$ is the interior of a convex polygon in $\mathbb{R}^{2}$ and $f,g$
two asymptotic geodesics, we show that the distance function $d\left(
f\left(  t\right)  ,g\left(  t\right)  \right)  $ is convex for $t$
sufficiently large. The same result is obtained in the case $\partial \Omega$
is of class $C^{2}$ and the curvature of $\partial \Omega$ at the point
$f\left(  \infty \right)  =g\left(  \infty \right)  $ does not vanish. An 
example is provided for the necessity of the curvature assumption. 

\end{abstract}

\section{Introduction and statements of results}

Let $\Omega$ be a bounded convex (open) domain in $\mathbb{R}^{n}$ and $h$ the
Hilbert metric on $\Omega$ defined as follows: for any distinct points $p,q$
in $\Omega$ let $p^{\prime}$ and $q^{\prime}$ be the intersections of the line
through $p$ and $q$ with $\partial \Omega$ closest to $p$ and $q$ respectively.
Then
\[
h\left(  p,q\right)  =\log \frac{\left \vert p^{\prime}-q\right \vert
\cdot \left \vert q^{\prime}-p\right \vert }{\left \vert p^{\prime}-p\right \vert
\cdot \left \vert q^{\prime}-q\right \vert }%
\]
where $\left \vert z-w\right \vert $ denotes the usual Euclidean distance. The
quantity $\frac{\left \vert p^{\prime}-q\right \vert \cdot \left \vert q^{\prime
}-p\right \vert }{\left \vert p^{\prime}-p\right \vert \cdot \left \vert q^{\prime
}-q\right \vert }$ is the cross ratio of the colinear points $p,q,q^{\prime
},p^{\prime}$ denoted by $\left[  p,q,q^{\prime},p^{\prime}\right]  $ and is
invariant under projective transformations of $\mathbb{R}^{n} .$ We refer to \cite{Bus}, \cite{Har}
and \cite{PaTr} for the basic properties of the distance $h$ as well as a
presentation of classic and contemporary aspects of Hilbert Geometry.

It is well known that, contrary to non-positively curved Riemannian geometry,
the distance between two points moving at unit speed along two geodesics is
not necessarily convex. The behavior near infinity of the distance function
\[
t\rightarrow h\left(  f\left(  t\right)  ,g\left(  t\right)  \right)
\]
when $f,g$ are two intersecting geodesics is studied in detail in \cite{Soc}.
In this note we are concerned with the case of asymptotic geodesics.

Each geodesic line $f$ determines two points at infinity denoted by $f\left(
-\infty \right)  $ and $f\left(  +\infty \right)  $ which are distinct points in
$\partial \Omega.$ Two geodesics $f,g$ are said to be \emph{asymptotic} if
$f\left(  +\infty \right)  =g\left(  +\infty \right)  .$ We first show that, up
to re-parametrization, the distance function $h\left(  f(t),g(t)\right)  $
tends to $0$ when $t\rightarrow \infty,$ provided that $\partial \Omega$ is
$C^1$ at the point $f\left(  +\infty \right)  =g\left(  +\infty \right)  .$
Then we show that, near infinity, the distance function is convex when
$\Omega$ is the interior of a convex polytope in $\mathbb{R}^{n}$ as well as
when $\Omega$ is a convex domain in $\mathbb{R}^{n}$ with $C^{2}$ boundary
such that the curvature of $\partial \Omega$ at $f(+\infty
)=g(+\infty)$ along the plane determined by $f$ and $g$ is not zero. The
precise statement is the following

\begin{theorem}
\label{mainth}Suppose $f,g$ are two asymptotic geodesic lines in a convex bounded
domain $\Omega$ with common boundary point $\xi=f(+\infty)=g(+\infty
)\in \partial \Omega$ and $P$ the plane determined by $f$ and $g.$ If $\Omega$
is either
\begin{itemize}
 \item[(a)] the interior of a convex polytope in $\mathbb{R}^{n},$ or
 \item[(b)] a convex domain in $\mathbb{R}^{n}$ with $C^{2}$ boundary and the
 curvature of $\partial \Omega$ at $\xi$ along $P$ is not zero,
\end{itemize}
then, there exists $T>0$ such that the function $t\rightarrow h\left(
f\left(  t\right)  ,g\left(  t\right)  \right)  $ is convex for $t>T.$
\end{theorem}

In the last Section and for the case of non-asymptotic geodesics, two simple examples are provided demonstrating non-convexity for either intersecting or, disjoint geodesics. Finally,
an example demonstrating the necessity of the curvature condition in Theorem \ref{mainth}b
above is provided (see Example 3).

Dynamical properties of the geodesic flow on $\Omega /\Gamma$ equipped with the Hilbert metric, where $\Gamma$ is a torsion free discrete group which divides the strictly   convex domain $\Omega$, have been studied by Y. Benoist (see \cite{Ben}) using the Anosov properties of the flow.
In view of Eberlein's approach in the study of the geodesic flow (see \cite{Ebe1}, \cite{Ebe2} and \cite{CPT}) which is based on the convexity of the distance function  as well as the zero distance of asymptotic geodesics, mixing of geodesic flow
in the Hilbert geometry setting can be established using Theorem \ref{mainth}b and Proposition \ref{dzero} below.

\section{Distance of asymptotic geodesics}

We will always work with a pair of geodesic lines which determine a plane $P$
in $\mathbb{R}^{n}.$ As the distance function only depends on the affine
section $P\cap \Omega$ we will assume for the rest of this paper that $\Omega$
is a bounded convex domain in $\mathbb{R}^{2}.$

Recall that the Euclidean line $\ell_{\xi}$ is called a \emph{support line}
for $\Omega$ at the point $\xi \in \partial \Omega$ if $\partial \Omega \cap
\ell_{\xi}\ni \xi$ and $\Omega \cap \ell_{\xi}=\varnothing.$ Note that if
$\partial \Omega$ is smooth at $\xi,$ then the support line is the unique
tangent line at $\xi.$

\begin{figure}[ptb]
\begin{center}
\includegraphics
[scale=0.8]
{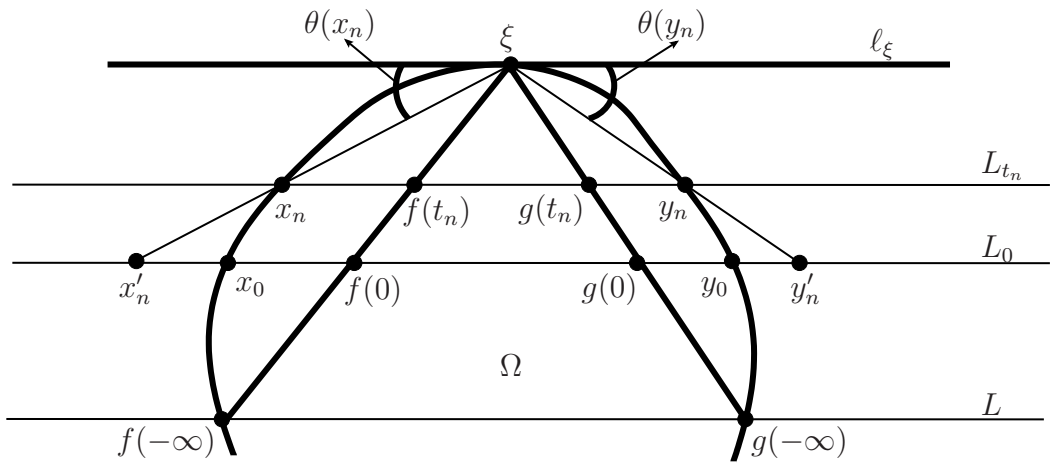}
\end{center}
\par
\begin{picture}(22,12)
\put(247,195){$\theta(y_n)$}
\put(121,195){$\theta(x_n)$}
\put(335,188){$\ell_{\xi}$}
\put(195,190){$\xi$}
\put(195,65){$\Omega$}
\put(290,37){$g(-\infty)$}
\put(50,37){$f(-\infty)$}
\put(377,50){$L$}
\put(377,110){$L_0$}
\put(377,142){$L_{t_n}$}
\put(51,95){$x^{\prime}_n$}
\put(136,94){$f(0)$}
\put(226,95){$g(0)$}
\put(110,125){$x_n$}\put(255,126){$y_n$}
\put(271,97){$y_0$}\put(95,97){$x_0$}
\put(305,95){$y^{\prime}_n$}
\put(202,125){$g(t_n)$}
\put(158,124){$f(t_n)$}
\end{picture}
\caption{Overview of notation for the proof of Proposition \ref{dzero}.}
\label{quardangle}
\end{figure}

\begin{proposition}
\label{dzero}Let $f,g$ be two asymptotic geodesic lines with
common boundary point $\xi=f(+\infty)=g(+\infty)\in \partial \Omega.$ Assume
that $\partial \Omega$ is $C^1$ at $\xi.$ Then there exists a (geodesic)
re-parametrization of $f$ such that
\[
\lim_{t\rightarrow \infty}h\left(  f(t),g(t)\right)  =0.
\]

\end{proposition}

\begin{proof}
Let $\ell_{\xi}$ be the tangent line at $\xi$ and $L$ the line containing
$f(-\infty)$ and $g(-\infty).$ We first treat the case where $\ell_{\xi}%
\cap \partial \Omega=\left \{  \xi \right \}  ,$ the other case being when
$\ell_{\xi}\cap \partial \Omega$ is an Euclidean segment containing $\xi.$

If $\ell_{\xi}\cap L$ is a (finite) point $A,$ we may compose with a
projective transformation which sends $A\ $to $\infty$ and, thus, we may
assume that $\ell_{\xi},L$ are parallel. For each $t\in \mathbb{R},$ using the
line $L_{t}$ containing $g(t)$ and parallel to $L$ we obtain a new
parametrization for $f$ by setting $f(t):=L_{t}\cap \mathrm{Im}f.$ For all
$t\in \mathbb{R},$ by similarity of the (Euclidean) triangles $\bigl(\xi
,f(t),g(t)\bigr)$ and $\bigl(\xi,f(0),g(0)\bigr),$ we have
\[
h(f(0),f(t))=h(g(0),g(t))=t
\]
hence $f$ is re-parametrized by arc length.

Pick an arbitrary sequence $\left \{  t_{n}\right \}  $ of positive reals
converging to infinity. For each $n\in \mathbb{N},$ the geodesic segment
$[f(t_{n}),g(t_{n})]$ determines two points in $\partial \Omega$ denoted by
$x_{n}$ and $y_{n}$ so that $\left[  x_{n},g\left(  t_{n}\right)  \right]  $
contains $f\left(  t_{n}\right)  $ and does not contain $y_{n}.$ For each
$n>0$ extend the (Euclidean) segment $[\xi,x_{n}]$ and denote by
$x_{n}^{\prime}$ its intersection with the line $L_{0}.$ All the above notation is displayed in Figure \ref{quardangle}.

Denote by $T_{n}$
(resp. $T_{n}^{\prime}$) the triangle with vertices $\xi,x_{n}$ and $y_{n}$
(resp. $x_{n}^{\prime}$ and $y_{n}^{\prime}$). 
Denote by $h_{T_{n}}$ and
$h_{T_{n}^{\prime}}$ the corresponding Hilbert metrics. We have
\[%
\begin{split}
h(f\left(  t_{n}\right)  ,g\left(  t_{n}\right)  )=h_{T_{n}}(f\left(
t_{n}\right)  ,  &  g\left(  t_{n}\right)  )=\log[f\left(  t_{n}\right)
,g\left(  t_{n}\right)  ,y_{n},x_{n}]\\
&  =\log[f\left(  0\right)  ,g\left(  0\right)  ,y_{n}^{\prime},x_{n}^{\prime
}]=d_{T_{n}^{\prime}}(f\left(  0\right)  ,g\left(  0\right)  ).
\end{split}
\]
Hence, it suffices to show that $h_{T_{n}^{\prime}}(f\left(  0\right)
,g\left(  0\right)  )\rightarrow0$ as $n\rightarrow \infty.$

As $\partial \Omega$ is smooth at $\xi$ and $\ell_{\xi}\cap \partial
\Omega=\left \{  \xi \right \}  $ the angle $\theta(x_{n})$ (resp. $\theta
(y_{n})$) formed by $\ell_{\xi}$ and the segment $[x_{n},\xi]$ (resp.
$[y_{n},\xi]$) is well defined. Clearly,
\[
\theta(x_{n})\rightarrow0\mathrm{\  \ and\  \ }\theta(y_{n})\rightarrow
0\mathrm{\  \ as\  \ }n\rightarrow \infty
\]
which implies that
\[
|x_{n}^{\prime}-f\left(  0\right)  |\rightarrow \infty \mathrm{\  \ and\  \ }%
|y_{n}^{\prime}-g\left(  0\right)  |\rightarrow \infty.
\]
Therefor, both fractions
\[%
\begin{split}
\frac{|x_{n}^{\prime}-g\left(  0\right)  |}{|x_{n}^{\prime}-f\left(  0\right)
|}=  &  \frac{|x_{n}^{\prime}-f\left(  0\right)  | + |f\left(  0\right)
-g\left(  0\right)  |}{|x_{n}^{\prime}-f\left(  0\right)  |}\mathrm{\ ,\ }\\
&  \frac{|y_{n}^{\prime}-f\left(  0\right)  |}{|y_{n}^{\prime}-g\left(
0\right)  |}=\frac{|y_{n}^{\prime}-g\left(  0\right)  |+|f\left(  0\right)
-g\left(  0\right)  |}{|y_{n}^{\prime}-g\left(  0\right)  |}%
\end{split}
\]
converge to $1$ and hence
\[
h_{T_{n}^{\prime}}(f\left(  0\right)  ,g\left(  0\right)  )=\log \left(
\frac{|x_{n}^{\prime}-g\left(  0\right)  |\, \,|y_{n}^{\prime}-f\left(
0\right)  }{|x_{n}^{\prime}-f\left(  0\right)  |\, \,|y_{n}^{\prime}-g\left(
0\right)  |}\right)  \rightarrow0.
\]
We now treat the case where $\xi$ is contained in a segment $\sigma
\subset \partial \Omega.$ We may assume that $\sigma$ and the line $L$
containing $f\left(  -\infty \right)  $ and $g\left(  -\infty \right)  $ are
parallel, otherwise, we may compose by a projective transformation sending the
intersection point at infinity. As in the previous case, define a new geodesic
parametrization of $f$ by setting $f(t):=L_{t}\cap \mathrm{Im}f$ where $L_{t}$
is the line containing $g(t)$ and parallel to $L.$ Pick a simple close $C^{1}$
curve $\tau$ with the following properties:

(1) $\tau$ bounds a convex domain $\Omega^{\prime}\subsetneq \Omega$

(2) $\tau$ contains $\xi,f\left(  -\infty \right)  ,g\left(  -\infty \right)  $

(3) the tangent line to $\tau$ at $\xi$ contains $\sigma.$\newline Denote by
$h^{\prime}$ the Hilbert distance in $\Omega^{\prime}.$ By the previous case,
\[
h^{\prime}(f\left(  t\right)  ,g\left(  t\right)  )\rightarrow0 \text{\ as\ }
t\rightarrow \infty.
\]
Since $\Omega^{\prime}\subset \Omega,$ we have
\[
h(f\left(  t\right)  ,g\left(  t\right)  )\leq h^{\prime}(f\left(
t\right)  ,g\left(  t\right)  )
\]
for all $t,$ which completes the proof of the proposition.\hfill
\end{proof}

\begin{remark}
If $\partial \Omega$ is not smooth at $\xi$ then the distance function is
bounded away from $0.$ To see this, pick two distinct support lines $\ell
_{1},\ell_{2}$ at $\xi.$ One of them, say $\ell_{1},$ intersects $L$ at a
point, say $A.$ Then, a fraction involving the sines of the angles formed by
the segments $\left[  A,\xi \right]  ,\left[  f\left(  -\infty \right)
,\xi \right]  $ and $\left[  g\left(  -\infty \right)  ,\xi \right]  $ at $\xi$
is a lower bound for the distance function.
\end{remark}

\begin{corollary}
Let $f,g$ be two asymptotic geodesic lines with common boundary point
$\xi=f(+\infty)=g(+\infty)\in \partial \Omega.$ Assume that $\partial \Omega$ is
$C^1$ at $\xi.$ Then 
\[ \lim_{t\rightarrow \infty}h\left(  f(t),g(t)\right)  =C\]
for some non-negative real $C.$
\end{corollary}
\begin{proof}
In the proof of Proposition \ref{dzero}, a new parametrization for $f$ was
defined. Denote by $\overline{f}$ the same geodesic line with the new
parametrization and let $C$ be the unique real number so that $f\left(
C\right)  =\overline{f}\left(  0\right)  .$ Clearly, $f\left(  t+C\right)  =$
$\overline{f}\left(  t\right)  ,$ hence,%
\[%
\begin{array}
[c]{lll}%
h\left(  f(t),g(t)\right)  & \leq & h\left(  f\left(  t\right)  ,f\left(
t+C\right)  \right)  +h\left(  f\left(  t+C\right)  ,g\left(  t\right)
\right) \\
& \leq & \left \vert C\right \vert +h\left(  \overline{f}\left(  t\right)
,g\left(  t\right)  \right)
\end{array}
\]
and
\[%
\begin{array}
[c]{lll}%
h\left(  f(t),g(t)\right)  & \geq & h\left(  f\left(  t\right)  ,f\left(
t+C\right)  \right)  -h\left(  f\left(  t+C\right)  ,g\left(  t\right)
\right) \\
& = & \left \vert C\right \vert -h\left(  \overline{f}\left(  t\right)
,g\left(  t\right)  \right)
\end{array}
\]
where $h\left(  \overline{f}\left(  t\right)  ,g\left(  t\right)  \right)
\rightarrow0$ as t$\rightarrow \infty.$ \hfill\end{proof}

\section{Convexity of the distance function}

We start with an elementary lemma concerning the Hilbert distance.

\begin{lemma}
\label{xift} Let $\Omega$ be a convex domain and $f$ a geodesic line in
$\Omega$ such that $\left \vert f\left(  -\infty \right)  -f\left(  0\right)
\right \vert =\left \vert f\left(  +\infty \right)  -f\left(  0\right)
\right \vert =1/2.$ Then for any $t>0$ the Euclidean distance $E\left(
t\right)  =\left \vert f\left(  t\right)  -f\left(  +\infty \right)  \right \vert
$ is given by the formula
\[
E\left(  t\right)  =\frac{1}{e^{t}+1}.
\]

\end{lemma}

The proof is straightforward by the definition of the Hilbert metric:%
\[
h\left(  f\left(  0\right)  ,f\left(  t\right)  \right)  =\log \frac{\left \vert
f\left(  -\infty \right)  -f\left(  t\right)  \right \vert \cdot \left \vert
f\left(  +\infty \right)  -f\left(  0\right)  \right \vert }{\left \vert f\left(
-\infty \right)  -f\left(  0\right)  \right \vert \cdot \left \vert f\left(
+\infty \right)  -f\left(  t\right)  \right \vert }=\log \frac{\left(  1-E\left(
t\right)  \right)  \cdot \frac{1}{2}}{\frac{1}{2} \cdot E\left(  t\right)  }
\]
hence $\displaystyle
e^{t}=e^{h\left(  f\left(  0\right)  ,f\left(  t\right)  \right)  }=
\frac{1-E\left(  t\right)  }{E\left(  t\right)  }$ and $\displaystyle E\left(
t\right)  =\frac{1}{e^{t}+1}.$\newline

We also need an elementary calculus lemma.

\begin{lemma}
\label{convy} There exists $T>0$ such that the following function is convex
\[
\phi:\left[  T,+\infty \right)  \rightarrow \mathbb{R}:\phi \left(  t\right)
=\log \frac{\beta+\left(  \alpha+\frac{1}{2}\right)  E\left(  t\right)  }%
{\beta+\left(  \alpha-\frac{1}{2}\right)  E\left(  t\right)  }%
\]
where $\alpha,\beta$ are real numbers with $\beta>0$ and $E\left(  t\right)
=\frac{1}{e^{t}+1}.$
\end{lemma}

\begin{figure}[ptb]
\begin{center}
\includegraphics
[scale=0.5]
{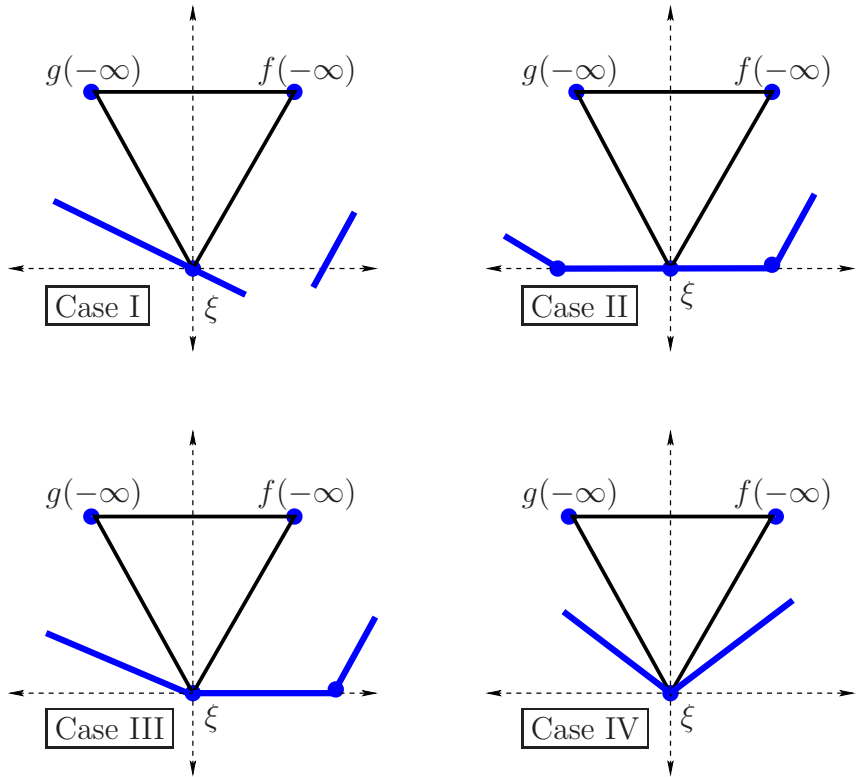}
\end{center}
\par
\begin{picture}(22,12)
\put(60,45){\fbox{Case III}}
\put(60,205){\fbox{Case I}}
\put(240,45){\fbox{Case IV}}
\put(240,205){\fbox{Case II}}
\put(120,50){$\xi$}
\put(120,205){$\xi$}
\put(300,50){$\xi$}
\put(300,210){$\xi$}
\put(60,135){$g(-\infty)$}
\put(60,295){$g(-\infty)$}
\put(240,135){$g(-\infty)$}
\put(240,295){$g(-\infty)$}
\put(140,135){$f(-\infty)$}
\put(140,295){$f(-\infty)$}
\put(320,135){$f(-\infty)$}
\put(320,295){$f(-\infty)$}
\end{picture}
\caption{The four generic positions of the sides of $\partial \Omega$ with
respect to $\xi$ and the coordinate axes (Theorem 1a).}\label{figcases}
\end{figure}

\begin{proof}
An elementary calculation shows that 
\[\displaystyle \phi^{\prime}\left(  t\right)=
\frac{\beta E^{\prime}\left(  t\right)}{\left[  \beta+\left(
\alpha+\frac{1}{2}\right)  E\left(  t\right)  \right] \left[
\beta+\left(  \alpha-\frac{1}{2}\right)  E\left(  t\right)  \right]}\]
and the second derivative of $\phi$ is
\[
\phi^{\prime \prime}\left(  t\right)  =\frac{\Phi}{\left[  \beta+\left(
\alpha+\frac{1}{2}\right)  E\left(  t\right)  \right]  ^{2}\left[
\beta+\left(  \alpha-\frac{1}{2}\right)  E\left(  t\right)  \right]  ^{2}}%
\]
where the numerator $\Phi$ is\\[3mm]
\hspace*{18mm} $\Phi=\beta \left[  \beta+\left(  \alpha+\frac{1}{2}\right)
E\left(  t\right)  \right]  \left[  \beta+\left(  \alpha-\frac{1}{2}\right)
E\left(  t\right)  \right]  E^{\prime \prime}\left(  t\right)  $ \\[2mm]%
\hspace*{43mm} $+\beta E^{\prime}\left(  t\right)  \left(  \alpha+\frac{1}%
{2}\right)  E^{\prime}\left(  t\right)  \left[  \beta+\left(  \alpha-\frac
{1}{2}\right)  E\left(  t\right)  \right]  $ \\[2mm]\hspace*{58mm} $-\beta
E^{\prime}\left(  t\right)  \left(  \alpha-\frac{1}{2}\right)  E^{\prime
}\left(  t\right)  \left[  \beta+\left(  \alpha+\frac{1}{2}\right)  E\left(
t\right)  \right]  .$\\[3mm]
As the asymptotic behavior of $E\left(  t\right)  ,E^{\prime}\left(  t\right)
$ and $E^{\prime \prime}\left(  t\right)  $ is $\frac{1}{e^{t}},\frac{-1}%
{e^{t}}$ and $\frac{1}{e^{t}}$ respectively, the dominant summand of $\Phi$
is
\[
\beta^{3}E^{\prime \prime}\left(  t\right)  =\beta^{3}\frac{e^{3t}-e^{t}%
}{\left(  e^{t}+1\right)  ^{4}}.
\]
Since $\beta>0,$ this completes the proof.
\end{proof}\\[3mm]
\textbf{Proof of Theorem \ref{mainth}(a).}

Let $\Omega$ be the interior of a convex polygon in $\mathbb{R}^{2}$ and $f,g$
two asymptotic geodesic lines with common boundary point $\xi=f(+\infty
)=g(+\infty)\in \partial \Omega.$ Recall that a projective transformation
preserves straight lines, convexity and cross ratio of four colinear points.
Moreover, a projective transformation is uniquely determined by its image on
four points provided that no three of them are colinear. As the latter
property is satisfied by the points $f(-\infty),$ $g(-\infty),$ $f(0)$ and
$g(0),$ we may assume, after composing by the appropriate projective
transformation that the coordinates of the four points mentioned above
are:\\[3mm]$f\left(  -\infty \right)  \equiv \left(  \frac{1}{2},\frac{\sqrt{3}%
}{2}\right)  ,g(-\infty)\equiv \left(  -\frac{1}{2},\frac{\sqrt{3}}{2}\right)
,f(0)\equiv \left(  \frac{1}{4},\frac{\sqrt{3}}{4}\right)  \text{\  \ and\  \ }%
g(0)\equiv \left(  -\frac{1}{4},\frac{\sqrt{3}}{4}\right)  .$\\[3mm]In
particular, the point $\xi=f(+\infty)=g(+\infty)$ is the point $\left(
0,0\right)  $ and the points $\xi,f\left(  -\infty \right)  ,g\left(
-\infty \right)  $ form an equilateral triangle with side length $1.$ There are four generic cases to examine depending on
whether the point $\xi \equiv \left(  0,0\right)  $ is a vertex of the polygon
$\partial \Omega$ or not and whether the side containing $\xi$ intersects the
$x-$axis only at $\xi$ or not. In Figure \ref{figcases} the thick segments represent sides of $\partial \Omega$ demonstrating the four cases. \begin{figure}[ptb]
\begin{center}
\includegraphics
[scale=1.2]
{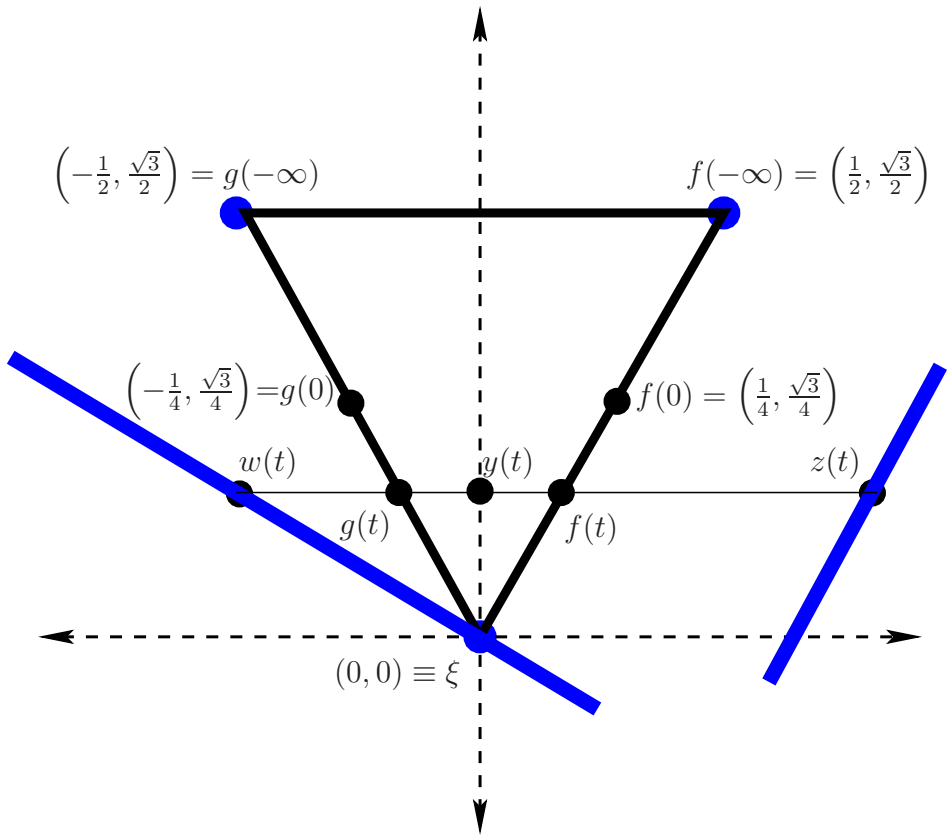}
\end{center}
\par
\begin{picture}(22,12)
\put(45,279){$\left(-\frac{1}{2},\frac{\sqrt{3}}{2}\right)=g(-\infty)$}
\put(285,279){$f(-\infty)=\left(\frac{1}{2},\frac{\sqrt{3}}{2}\right)$}
\put(72,197){$\left(-\frac{1}{4},\frac{\sqrt{3}}{4}\right)$\hspace*{-1mm} =$g(0)$}
\put(266,195){$f(0)=\left(\frac{1}{4},\frac{\sqrt{3}}{4}\right)$}
\put(116,170){$w(t)$}
\put(332,170){$z(t)$}
\put(208,170){$y(t)$}
\put(239,145){$f(t)$}
\put(154,146){$g(t)$}
\put(152,90){$(0,0)\equiv \xi$}
\end{picture}
\label{isosc}\caption{The calculation of the Hilbert distance $h\left(
f(t),g(t)\right)  $ (Theorem 1a)}\label{isosc1a}
\end{figure}

Clearly, in Case IV the distance $h\left(  f(t),g(t)\right)  $ is constant
because for all sufficiently large $t\neq t^{\prime}$ the lines containing
$f(t),g(t)$ and $f(t^{\prime}),g(t^{\prime})$ are parallel. We will deal in
detail with Case I and the arguments will suffice for the remaining cases II
and III.

By Lemma \ref{xift} we have%
\[
\left \vert f\left(  t\right)  -\xi \right \vert =\left \vert g\left(  t\right)
-\xi \right \vert =\frac{1}{e^{t}+1}\equiv E(t).
\]
Let $\left(  0,y\left(  t\right)  \right)  $ be the intersection point of the
line containing $f\left(  t\right)  ,g\left(  t\right)  $ with the $y-$axis.
As the triangle formed by $\xi,g\left(  t\right)  ,f\left(  t\right)  $ is
equilateral we have
\[
y\left(  t\right)  =E\left(  t\right)  \frac{\sqrt{3}}{2}=\frac{\sqrt{3}%
/2}{e^{t}+1}.
\]
Then the unique side of $\partial \Omega$ not containing $\xi$ and intersecting
the $x-$axis has the form%
\[
z\left(  t\right)  =\left(  \alpha y\left(  t\right)  +\beta,y\left(
t\right)  \right)
\]
and the side containing $\xi$ has the form%
\[
w\left(  t\right)  =\left(  - \alpha^{\prime}y\left(  t\right)  ,y\left(
t\right)  \right)
\]
for $\alpha \in \mathbb{R},$ $\beta>0$ and $\alpha^{\prime}>\frac{\sqrt{3}}{3}
.$ The latter holds because the angle formed by the segments $\left[
\xi,w\left(  t\right)  \right]  $ and $\left[  \xi,\left(  0,y\left(
t\right)  \right)  \right]  $ at $\xi$ belongs to $\left( \frac{2\pi}{3}
,\pi \right) .$ All the above notation is visualized in  Figure \ref{isosc1a}.

We next compute the Euclidean distances involved in the definition of the
distance $h\left(  f\left(  t\right)  ,g\left(  t\right)  \right)  :$%
\[%
\begin{array}
[c]{ccl}%
\left \vert z(t)-f\left(  t\right)  \right \vert  & = & \alpha y\left(
t\right)  +\beta-\frac{1}{2}\left \vert g\left(  t\right)  -f\left(  t\right)
\right \vert \, \,=\\[2mm]
& = & \alpha \frac{\sqrt{3}/2}{e^{t}+1}+\beta-\frac{1}{2}\frac{1}{e^{t}%
+1}=\left(  \frac{\alpha \sqrt{3}}{2}-\frac{1}{2}\right)  E\left(  t\right)
+\beta \\[2mm]%
\left \vert z(t)-g\left(  t\right)  \right \vert  & = & \left(  \frac
{\alpha \sqrt{3}}{2}+\frac{1}{2}\right)  E\left(  t\right)  +\beta \\[2mm]%
\left \vert w(t)-g\left(  t\right)  \right \vert  & = & \alpha^{\prime}y\left(
t\right)  -\frac{1}{2}\left \vert g\left(  t\right)  -f\left(  t\right)
\right \vert \, \, =\\[2mm]
& = & \alpha^{\prime}\frac{\sqrt{3}/2}{e^{t}+1}-\frac{1}{2}\frac{1}{e^{t}%
+1}=\left(  \frac{\alpha^{\prime}\sqrt{3}}{2}-\frac{1}{2}\right)  E\left(
t\right) \\[2mm]%
\left \vert w(t)-f\left(  t\right)  \right \vert  & = & \left(  \frac
{\alpha^{\prime}\sqrt{3}}{2}+\frac{1}{2}\right)  E\left(  t\right)
\end{array}
\]
It follows that
\[
h\left(  f\left(  t\right)  ,g\left(  t\right)  \right)  =\log \frac{\left(
\frac{\alpha \sqrt{3}}{2}+\frac{1}{2}\right)  E\left(  t\right)  +\beta
}{\left(  \frac{\alpha \sqrt{3}}{2}-\frac{1}{2}\right)  E\left(  t\right)
+\beta}+\log \biggl(A^{\prime}E\left(  t\right) \biggr)  \text{\  \ where\  \ }%
A^{\prime}=\frac{\frac{\alpha^{\prime}\sqrt{3}}{2}+\frac{1}{2}}{\frac
{\alpha^{\prime}\sqrt{3}}{2}-\frac{1}{2}}.
\]
The second summand is convex for all $t$ because $\alpha^{\prime} >
\frac{\sqrt{3}}{3},$ hence, $A^{\prime}>0.$ There exists, by Lemma
\ref{convy}, $T>0$ such that the first summand is convex for $t>T$ as
required. This completes the proof of part 1a. \\[2mm]
\textbf{Proof of Theorem \ref{mainth}(b).}  Let $\Omega$ be a
convex domain in $\mathbb{R}^{2}$ with $C^{2}$ boundary and  $f,g$  two asymptotic geodesic lines with common boundary point $\xi=f(+\infty)=g(+\infty)\in \partial \Omega.$ We will need the following well known Lemma whose proof is included for the reader's convenience.
\begin{lemma}
\label{curvat} Let $p$ be a projective transformation of $\mathbb{R}^{2} .$
sending $\partial \Omega$ to a bounded curve. If the curvature of $\partial \Omega$ at the point $\xi \in \partial \Omega$ is not zero then the same holds for the point $p(\xi) \in p \left( \partial (\Omega) \right) .$
\end{lemma}
\textbf{Proof of Lemma.}
Identify  $\mathbb{R}^{2} $ with the plane 
$\Pi =\left\{ (x,y,z)|z=1  \right\}$ in the real projective space 
$\mathbb{R}P^{2} =\mathbb{R}^{3}- \{ (0,0,0)\} / \sim $ whose points are rays emanating from the origin. The image of  $\partial \Omega$ under a projective transformation can be taken as the composition of
\begin{itemize}
 \item an invertible linear transformation of $\mathbb{R}^3,$ and
 \item the projection of $A\left(\partial \Omega\right)\subset A\left(\Pi\right)$ onto $\Pi$ along the rays through the origin.
\end{itemize}
Since an invertible linear transformation of $\mathbb{R}^3$ sends $C^2$ curves to $C^2$ curves and preserves the non-vanishing curvature property, it suffices to check the desired property for the projection $A\left(\Pi\right)\longrightarrow \Pi .$

To see this, let $E_1 , E_2$ be two hyperplanes  intersecting a $C^2$ 
cone $K$ through the origin and denote by $\sigma_i $ the simple closed convex $C^2$ curve determined by the intersection $E_i \cap K, i=1,2.$ Moreover, as $K$ is convex, the curvature of  $\sigma_i $ at any point is 
$\geq 0. $ Let $\ell$ be a line through the origin contained in $K$ intersecting $\sigma_i$ at the point $\xi _i , i=1,2$ and $E^{\prime}_2$ the hyperplane containg  
$\xi_2$ and parallel to $E_1.$ Denote by $\kappa (\xi _i)$  the curvature of $\sigma_i$ at $\xi_i$ and by $\kappa^{\prime} (\xi _2)$ the curvature of
the curve
$E^{\prime}_2 \cap K$ at $\xi_2 .$ Assume $\kappa (\xi _1)\neq 0$ and,  clearly,  $\kappa^{\prime} (\xi _2) \neq 0. $ For, if $\kappa (\xi _2)= 0$
 then, since the line $\ell$ contains $\xi_2 ,$ both principal curvature at 
 $\xi_2$ would be $0$ contradicting the fact that  $\kappa^{\prime} (\xi _2) \neq 0. $    \hfill \rule{0.5em}{0.5em} \\[2mm]
\begin{figure}[ptb]
\begin{center}
\includegraphics
[scale=0.5]
{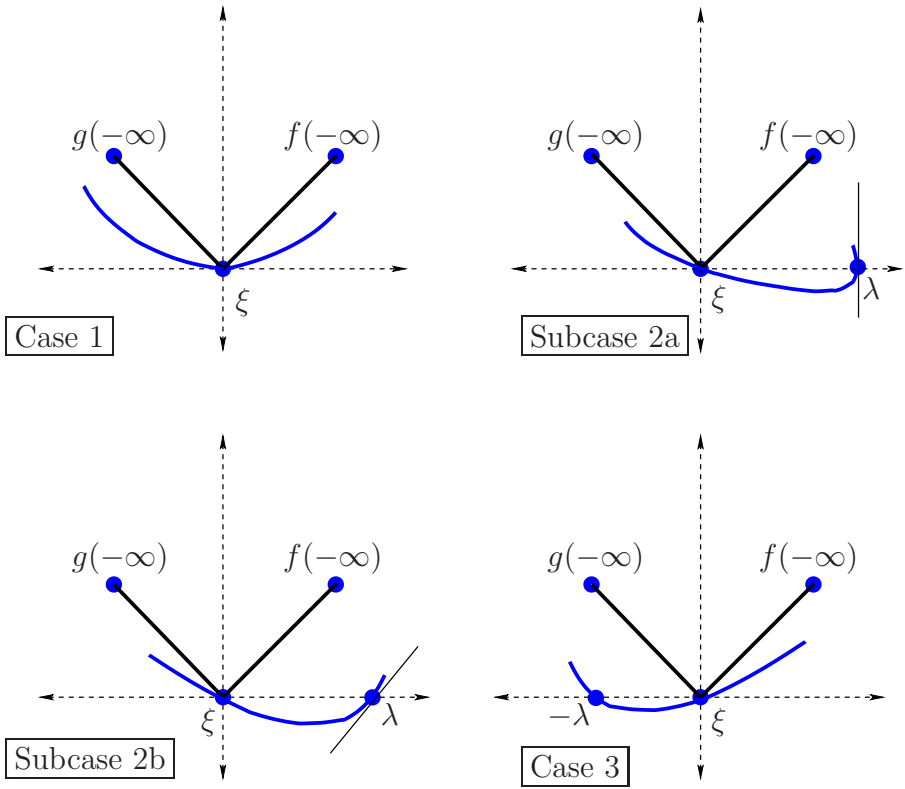}
\end{center}
\par
\begin{picture}(22,12)
\put(30,35){\fbox{Subcase 2b}}
\put(30,195){\fbox{Case 1}}
\put(225,32){\fbox{Case 3}}
\put(225,195){\fbox{Subcase 2a}}
\put(104,50){$\xi$}
\put(117,209){$\xi$}
\put(297,50){$\xi$}
\put(297,210){$\xi$}
\put(55,112){$g(-\infty)$}
\put(55,272){$g(-\infty)$}
\put(235,112){$g(-\infty)$}
\put(235,52){$-\lambda$}\put(172,52){$\lambda$}\put(354,213){$\lambda$}
\put(235,272){$g(-\infty)$}
\put(135,112){$f(-\infty)$}
\put(135,272){$f(-\infty)$}
\put(315,112){$f(-\infty)$}
\put(315,272){$f(-\infty)$}
\end{picture}
\caption{Cases of intersection of $\partial\Omega$ with the $x-$axis (Theorem 1b).}\label{fig12cases}
\end{figure}
Returning to the proof of Theorem \ref{mainth}(b),  the points $f\left(
-\infty \right)  ,g(-\infty),f(0)$\  \ and\  \ $g(0)$ form a non trivial
quadrilateral. After composing by a projective transformation we may assume
that the coordinates of the four points mentioned above are:
\begin{equation*}\begin{split}
f\left(-\infty\right)\equiv\left(\frac{\sqrt{2}}{2},\frac{\sqrt{2}}{2}\right),
g(-\infty)\equiv \left(  -\frac{\sqrt{2}}{2},\frac{\sqrt{2}}{2}\right) ,
& f(0) \equiv \left(  \frac{\sqrt{2}}{4},\frac{\sqrt{2}}{4}\right) 
\\ & \text{\  \ and\  \ }
g(0)\equiv \left(  -\frac{\sqrt{2}}{4},\frac{\sqrt{2}}{4}\right).
\end{split}\end{equation*}
In particular, the point $\xi=f(+\infty)=g(+\infty)$ is the point $\left(
0,0\right)  $ and the points $\xi,f\left(  -\infty \right)  ,g\left(
-\infty \right)  $ form a right equilateral triangle.
There are three cases to examine depending on the intersection of $\partial\Omega$ with the $x-$axis:

Case 1: $\partial\Omega \cap \left\{y=0\right\} = \left\{\xi\right\}.$

Case 2: $\partial\Omega \cap \left\{y=0\right\} = 
       \left\{\xi,\left(\lambda,0\right)\bigm\vert \lambda>0 \right\}.$ 
       
Case 3: $\partial\Omega \cap \left\{y=0\right\} = 
       \left\{\xi,\left(\lambda,0\right)\bigm\vert \lambda<0 \right\}.$  \\    
In Figure \ref{fig12cases}, these cases are demonstrated with the additional 
consideration, in Case 2, of two subcases (2a and 2b defined below) depending on the tangent line at the point $\left(\lambda,0\right) .$

       The Euclidean line containing $f(t)$ and $g(t)$ intersects the $y-$axis at the point $\left(0,y(t)\right)$ where,  by Lemma \ref{xift},
$\left| f(t) -\xi \right|=\frac{1}{e^{t}+1},$ hence, 
\begin{equation}y(t)=\frac{\sqrt{2}/2}{e^{t}+1}.\label{ytlength}\end{equation}
We also have that when $t\rightarrow+\infty$
\begin{equation}%
\begin{split}
ye^{t}=\frac{\sqrt{2}/2}{e^{t}+1}e^t\longrightarrow \sqrt{2}/2,\, \, \, \,y^{\prime}e^{t}= &
\frac{-(\sqrt{2}/2)e^{t}}{\left(  e^{t}+1\right)  ^{2}}e^{t}\longrightarrow
-\sqrt{2}/2\text{\ and\ }\\
&  y^{\prime \prime}e^{t}=\frac{\sqrt{2}}{2}\frac{e^{2t}-e^{t}}{\left(  e^{t}+1\right)  ^{3}%
}e^{t}\longrightarrow \sqrt{2}/2
\end{split}
\label{e2t}%
\end{equation}
The Euclidean line containing $f(t)$ and $g(t)$ also intersects $\partial\Omega$ at two points with coordinates
$\left(x(t),y(t)\right) ,x(t)>0$ and $\left(\overline{x}(t),y(t)\right) ,\overline{x}(t)<0.$ For $t$ large enough, $y$ is a function of $x,$ say, $y(t)=K\left(  x(t)\right)  $ for some 1-1 and $C^{2}$ function $K$ and, hence, $x$ is a function of $y,$ namely, $x(t)=K^{-1}\left( y(t)\right)  .$ Similarly, for $t$ large enough, $y(t)=\overline{K}\left(  \overline{x}(t)\right)  $ for some 1-1 and $C^{2}$ function $\overline{K}$ and $\overline{x}$ is a function of $y,\overline{x}(t)=\overline{K}^{-1}\left( y(t)\right)  .$\\
The Hilbert distance of $f(t),g(t)$ is given by
\begin{equation}
h\left(  f\left(  t\right)  ,g\left(  t\right)  \right)  =\log \frac
{x(t)+y(t)}{x(t)-y(t)}+\log \frac{\left|\overline{x}(t)\right|+y(t)}{\left|\overline{x}(t)\right| -y(t)}%
\label{hdist}%
\end{equation}
and we denote by $\phi (t)$ and $\overline{\phi} (t)$ the first and second summand respectively. We re-write the above mentioned Cases using the notation just introduced:\\[4mm]
\hspace*{2cm}\begin{minipage}{12cm}\begin{enumerate}
 \item[Case 1:] 
 Both $x(t), \overline{x}(t)\longrightarrow 0$ as $t\rightarrow +\infty .$
  
 \item[Case 2:] $x(t) \rightarrow \lambda >0$ and $ \overline{x}(t)\longrightarrow 0$ as $t\rightarrow +\infty .$
 
 \item[Case 3:]  $x(t) \rightarrow0$ and $ \overline{x}(t)\longrightarrow  \lambda^{\prime}<0$ as $t\rightarrow +\infty .$
\end{enumerate} \end{minipage}\\[4mm]
It suffices to deal only with the convexity of $\phi (t)$ in all three cases:
in Case 1 the proof for the convexity of $\overline{\phi}$ is identical with that of $\phi$ and convexity of $\overline{\phi}$ in Case 2 (resp. Case 3) follows from convexity of $\phi$ in Case 3 (resp. Case 2).

We will suppress the parameter $t$ and we will be writing $\frac
{dy}{dx}$ instead of $\frac{dK}{dx}$ and $\frac{dx}{dy}$ instead of
$\frac{d\left(K^{-1}\right)}{dy}.$ By the following calculation 
\[
 \begin{split} 
 \frac{d^{2}x}{dy^{2}} =  \frac{d}{dy}\left(  \frac{dx}{dy}\right)
 = \frac{d}{dy}\left(  \frac{1}{\left(  \frac{dy}{dx}\right)}\right)
 = -\frac{1}{\left(  \frac{dy}{dx}\right) ^2} \frac{d^{2}y}{dx^{2}}\frac{dx}{dy}
 =- \frac{d^{2}y}{dx^{2}} \left(  \frac{dy}{dx}\right) ^3
 \end{split}
\]
we have the formula 
\begin{equation}
\frac{d^{2}x}{dy^{2}}=-\frac{d^{2}y}{dx^{2}}\left(  \frac{dx}{dy}\right)
^{3}\label{fir}%
\end{equation}
First and second derivatives of $\phi(t)$ are as follows:
\[
\phi^{\prime}(t)=2 \left(  \frac{y}{x}\right)  ^{\prime}
           \frac{1}{1-\left(\frac{y}{x}\right)^2}
\text{\  \ and\  \ }\phi^{\prime \prime}(t)=2\left(  \frac{y}{x}\right)
^{\prime \prime}\frac{1}{1-\left(\frac{y}{x}\right)^2}+
2 \left(\frac{y}{x}\right)^{\prime} 
\frac{2\frac{y}{x}\left(  \frac{y}{x}\right)  ^{\prime}}{\left(1-\left(\frac{y}{x}\right)^2 \right)^2}.
\]
As the slope of the geodesic line $f$ is $1,$ it suffices to show that for $t$
large enough $\left(  \frac{y}{x}\right)  ^{\prime \prime}>0.$  We have the
following calculations
\begin{subequations}
\begin{align}
\left( \frac{y}{x}\right)  ^{\prime} =
  \frac{y^{\prime}x-yx^{\prime}}{x^{2}}=
\frac{y^{\prime}x-\frac{dx}{dy}
y^{\prime}y}{x^{2}}\,\,\,=\,\,\,\frac{x-\frac{dx}{dy}y}{x^{2}} y^{\prime}
\nonumber \\[3mm]
\left(  \frac{y}{x}\right)  ^{\prime \prime}\! \!x^{2}\frac{dy}{dx}=
\left[-\frac{dy}{dx}\frac{d^{2}x}{dy^{2}}y-2+2\frac{y}{x}\frac{dx}{dy}\right]
\left(y^{\prime}\right)^{2}+\left(  x\frac{dy}{dx}-y\right)y^{\prime \prime}
\label{first} \\[3mm]
\left(  \frac{y}{x}\right)  ^{\prime \prime}\! \!x^{2}\frac{dy}{dx}
  \stackrel{by\ (\ref{fir})}{=}  
 \left[ \left(\frac{dy}{dx}\right)^{-2}\frac{d^{2}y}{dx^{2}}y-2+2\frac{y}{x}\frac{dx}{dy}\right]
\left(y^{\prime}\right)^{2}+\left(  x\frac{dy}{dx}-y\right)y^{\prime \prime}
\label{second}
\end{align}
\end{subequations}
\underline{Case 1}: In this case the $x-$axis is the tangent line to $\partial \Omega$ at $\xi = (0,0)$ and using Lemma \ref{curvat} and our curvature hypothesis we have that
\[
\left.  \frac{d^{2}y}{dx^{2}}\right \vert _{x=0}\neq0 \text{\ \ and\ \ }
\left.  \frac{dy}{dx}\right \vert _{x=0}=0.
\]
Multiplying both sides of equation (\ref{second}) by $e^{2t}$ we have
\begin{equation}
\left(  \frac{y}{x}\right)  ^{\prime \prime}\! \!x^{2}\frac{dy}{dx}e^{2t}
              =
\left[  \vphantom{\left(\frac{dy}{dx}\right)^{-2}\frac{d^2 x}{dy^2}} \right.
\underbrace{\left(\frac{dy}{dx}\right)^{-2}\frac{d^{2}y}{dx^{2}}y}_{\underline{\text{term1}}}
                     -2+2
\underbrace{\frac{y}{x}\frac{dx}{dy}}_{\underline{\text{term2}}}\left.
\vphantom{\left(\frac{dy}{dx}\right)^{-2} \frac{d^2 x}{dy^2}y}\right] 
              \left(y^{\prime}\right)^{2}e^{2t}+
\left( \vphantom{ x\frac{dy}{dx}e^t -ye^t}\right.
\underbrace{x\frac{dy}{dx}e^{t}}_{\underline{\text{term3}}}-ye^{t}\left.
\vphantom{ x\frac{dy}{dx}e^t -ye^t}\right)  y^{\prime \prime}e^{t}%
.\label{basiceq}%
\end{equation}
We will show that {\underline{\text{term1}}} and {\underline{\text{term2}}}
both converge to $1/2$ as $t\rightarrow+\infty$ and {\underline{\text{term3}}}
converges to $2.$ Then using (\ref{e2t}) the right hand side of (\ref{basiceq}) converges to 
$\left(  \frac{1}{2}-2+2\frac{1}{2}\right)\left(-\sqrt{2}/2\right)^{2} +
\left(\sqrt{2}-\sqrt{2}/2\right)\left(\sqrt{2}/2\right)=\frac{1}{4}>0.$ 
This shows that $\left(  \frac{y}{x}\right)^{\prime \prime}>0$ 
which in turn implies that $\phi(t)$ is
convex for large enough $t.$ In the following calculations limits are always
taken as $t\rightarrow+\infty$ or, equivalently, $x,y\rightarrow0$ and the
symbol $\sim$ between two functions indicates that the their limits as
$t\rightarrow+\infty$ are equal.
\[
{\underline{\text{term1}}}=
\frac{d^{2}y}{dx^{2}}\frac{y}{\left(  \frac{dy}{dx}\right)^{2}}
\sim 
\frac{d^{2}y}{dx^{2}}\frac{y^{\prime}}{2\frac{dy}{dx}\left( \frac{dy}{dx}\right)^{\prime}}
=
\frac{d^{2}y}{dx^{2}}\frac{y^{\prime}}{2\frac{dy}{dx}\frac{d^{2}y}{dx^{2}}\frac{dx}{dy}y^{\prime}} 
\longrightarrow \frac{1}{2}.
\]
The more tedious calculation for {\underline{\text{term2}}} is as follows:
\begin{equation}
\begin{split}
 {\underline{\text{term2}}}=\frac{y}{x}\frac{dx}{dy}
\sim &
\frac{\left(\frac{y}{x}\right)^{\prime}}{\left(\frac{dy}{dx}\right)^{\prime}} 
=
\frac{\left(x-y\frac{dx}{dy}\right)y^{\prime}x^{-2}}{\frac{d^{2}y}{dx^{2}}\frac{dx}{dy}y^{\prime}}
=
\left(\frac{d^{2}y}{dx^{2}}\right)^{-1}\frac{x\frac{dy}{dx}-y}{x^{2}}
\\ & 
\begin{split}
\sim 
\left(\frac{d^{2}y}{dx^{2}}\right)^{-1} & \frac{\left(x\frac{dy}{dx}-y\right)  ^{\prime}}{\left(  x^{2}\right) ^{\prime}}
=\left(  \frac{d^{2}y}{dx^{2}}\right)  ^{-1}\frac{
\frac{dx}{dy}y^{\prime}\frac{dy}{dx}+x\frac{d^{2}%
y}{dx^{2}}\frac{dx}{dy}y^{\prime}-y^{\prime}}{2x\frac{dx}{dy}y^{\prime}%
} \\ & =
\left(  \frac{d^{2}y}{dx^{2}}\right)  ^{-1}\frac{1}{2}\frac{d^{2}y}{dx^{2}%
}=\frac{1}{2}.\end{split}
\end{split}
\end{equation}
For the calculation of {\underline{\text{term3}}} first observe that
\begin{equation}
\frac{d^{2}y}{dx^{2}}x\frac{dx}{dy}\sim \frac{d^{2}y}{dx^{2}}\frac{x^{\prime}%
}{\left(  \frac{dy}{dx}\right)  ^{\prime}}=\frac{d^{2}y}{dx^{2}}\frac
{\frac{dx}{dy}y^{\prime}}{\frac{d^{2}y}{dx^{2}}\frac{dx}{dy}y^{\prime}%
}=1.\label{help3}%
\end{equation}
Then we have
\[%
\begin{split}
{\underline{\text{term3}}}=x\frac{dx}{dy}e^{t}\sim &  \frac{\left(  x\frac
{dx}{dy}\right)  ^{\prime}}{\left(  e^{-t}\right)  ^{\prime}}=\frac{\frac
{dx}{dy}y^{\prime}\frac{dy}{dx}+x\frac{d^{2}y}{dx^{2}}\frac{dx}{dy}y^{\prime}%
}{-e^{-t}}\\
&  =-y^{\prime}e^{t}\left(  1+x\frac{d^{2}y}{dx^{2}}\frac{dx}{dy}\right)
\overset{by\ (\ref{e2t}),(\ref{help3})}{\longrightarrow}
-\left(-\sqrt{2}/2\right)(1+1)=\sqrt{2}.
\end{split}
\]
This completes the proof for the convexity of $\phi$ in Case 1.\\[3mm]
\underline{Case 2}: In this case the $x(t)\rightarrow \lambda>0$ as 
$t\rightarrow +\infty$ and we have two sub-cases depending on whether 
\begin{equation}
\left.  \frac{dx}{dy}\right \vert _{y=0}=0 \text{\ equivalently\ }
\left.  \frac{dy}{dx}\right \vert _{x=\lambda}=\infty\tag{Subcase 2a}
\label{s2a}
\end{equation}
or\begin{equation}
   \left.  \frac{dx}{dy}\right \vert _{y=0}\neq 0 \neq 
   \left.  \frac{dy}{dx}\right \vert _{x=\lambda}\tag{Subcase 2b}\label{s2b}
  \end{equation}
For \ref{s2a}, multiply both sides of equation (\ref{second}) by 
$e^{t}\frac{dx}{dy}$ to get
\begin{equation}
\left(  \frac{y}{x}\right)  ^{\prime \prime}\! \!x^{2}e^{t}
              =
\left[  
\left(\frac{dy}{dx}\right)^{-3}\frac{d^{2}y}{dx^{2}}y
-2 \frac{dx}{dy}+2\frac{y}{x}\left(\frac{dx}{dy} \right)^2
\right]
              \left(y^{\prime}\right)^{2}e^{t}
              +
\left( x-y\frac{dx}{dy} \right) y^{\prime \prime}e^{t}
.\label{basiceq2}
\end{equation}
By (\ref{e2t}),  $y^{\prime \prime}e^{t} \rightarrow \sqrt{2}/2$ and by assumptions in this Subcase, $\left( x-y\frac{dx}{dy} \right) \rightarrow
\lambda .$  Moreover, the quantity inside the square bracket is easily seen to be bounded and, since by (\ref{e2t}) $\left(y^{\prime}\right)^{2}e^{t}
\rightarrow 0 ,$  it follows that 
\[
\left(  \frac{y}{x}\right)  ^{\prime \prime}\! \!x^{2}e^{t} 
\longrightarrow \lambda\frac{\sqrt{2}}{2}
\]
hence, $\left(  \frac{y}{x}\right)  ^{\prime \prime}
$ is positive for large enough $t.$\\[3mm]
For \ref{s2b}, observe that $ \left.  \frac{dy}{dx}\right \vert _{x=\lambda}$
may be negative. For this reason, we multiply both sides of equation (\ref{second}) by $e^{t}$ to get
\begin{equation}
\left(  \frac{y}{x}\right)  ^{\prime \prime}\! \!x^{2}\frac{dy}{dx}e^t
                 =
\left[  
\left(\frac{dy}{dx}\right)^{-2}\frac{d^{2}y}{dx^{2}}y
-2 +2\frac{y}{x}\frac{dx}{dy} 
\right]
              \left(y^{\prime}\right)^{2}e^{t}
              +
\left( x\frac{dy}{dx}-y \right) y^{\prime \prime}e^{t}
.\label{basiceq3}
\end{equation}
In a similar manner as in the previous subcase we obtain  
\[
\left(  \frac{y}{x}\right)  ^{\prime \prime}\! \!x^{2}\frac{dy}{dx}e^{t} 
\longrightarrow \lambda\frac{dy}{dx}\frac{\sqrt{2}}{2}.
\]
This completes the proof of the convexity of $\phi$ in Case 2.\\[3mm]
\underline{Case 3}: In this case $x(t)\rightarrow 0$ as 
$t\rightarrow +\infty$ and $\left.\frac{dy}{dx}\right\vert_{x=0}\in (0,1)$
because the slope of the geodesic line $f$ is $1.$ We have the following preliminary calculations
\begin{figure}[ptb]
\begin{center}
\includegraphics
[scale=1.2]
{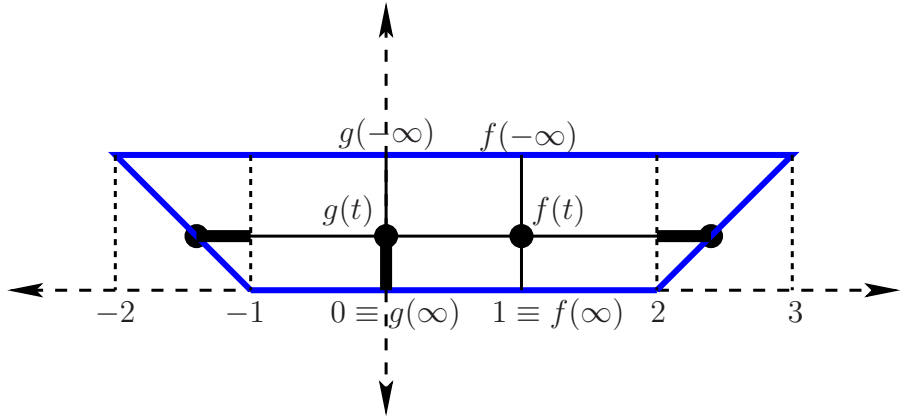}
\end{center}
\par
\begin{picture}(22,12)
\put(215,135){$f(-\infty)$}
\put(162,136){$g(-\infty)$}
\put(235,107){$f(t)$}
\put(156,107){$g(t)$}
\put(119,68){$-1$}
\put(70,68){$-2$}
\put(159,68){$0\equiv g(\infty)$}
\put(220,68){$1\equiv f(\infty)$}
\put(280,68){$2$}
\put(332,68){$3$}
\end{picture}
\caption{The thick segments are equal of Euclidean length $E(t)=\frac{1}%
{e^{t}+1}.$}\label{example1}
\end{figure}
\begin{equation}
 \begin{split}
   \frac{x\frac{dy}{dx}-y}{e^{-2t}}
 \sim &
 \frac{x\frac{d^2y}{dx^2}\frac{dx}{dy}y^{\prime}}{-2e^{-2t}}
=
 -\frac{1}{2}\,\, \frac{d^2y}{dx^2}\,\,\frac{dx}{dy}\,\, 
 \left(e^t x\right) \left(e^t y^{\prime}\right) \\ &
  \sim
  -\frac{1}{2}\,\, \frac{d^2y}{dx^2}\,\,\frac{dx}{dy}\,\, 
 \left(\frac{dx}{dy}\frac{\sqrt{2}}{2}\right) \left(-\frac{\sqrt{2}}{2}\right)
 =\frac{1}{4}\frac{d^2y}{dx^2}\left(\frac{dx}{dy}\right)^2 >0
 \end{split}\label{mple}
\end{equation}
where we used that $e^t x\sim \frac{dx}{dy}\frac{\sqrt{2}}{2}$ and 
$e^ty^{\prime}\rightarrow - \frac{\sqrt{2}}{2} .$
In a similar manner we obtain
\begin{equation}
   \frac{x-\frac{dx}{dy}y}{x^2} \sim 
   -\frac{1}{2} \left(\frac{dy}{dx}\right)^2\frac{d^2x}{dy^2}
   \label{mauro}
\end{equation}
We now multiply both sides of equation (\ref{first}) by $e^{3t}$ to get
\begin{equation}
\left(  \frac{y}{x}\right)  ^{\prime \prime}\! \!x^{2}\frac{dy}{dx}e^{3t}
                 =
\left[  
-\frac{dy}{dx}\frac{d^{2}x}{dy^{2}}ye^t
+2\frac{-1+\frac{y}{x}\frac{dx}{dy}}{e^{-t}} 
\right]
              \left(y^{\prime}\right)^{2}e^{2t}
              +
\frac{x\frac{dy}{dx}-y}{e^{-2t}} y^{\prime \prime}e^{t}
.\label{basiceq4}
\end{equation}
By (\ref{mple}) and the fact $\left(y^{\prime}\right)^{2}e^{2t}
\rightarrow \left(\sqrt{2}/2\right)^2,$ it suffices to show that the term in the square bracket converges to $0$ as $t\rightarrow \infty.$ 
For the first summand inside the square bracket we have
\begin{equation}
-\frac{dy}{dx}\frac{d^{2}x}{dy^{2}}ye^t \rightarrow 
  -\left.\frac{dy}{dx}\right\vert_{0}\frac{d^{2}x}{dy^{2}} \frac{\sqrt{2}}{2}
  \label{paralast}
  \end{equation}
For the second summand inside the square bracket we have
\begin{equation}
 \begin{split}
2\frac{-1+\frac{y}{x}\frac{dx}{dy}}{e^{-t}} &
\sim 
 2\frac{\left(  \frac{y}{x}\right)  ^{\prime } \frac{dx}{dy} 
        +\frac{y}{x}\frac{d^{2}x}{dy^{2}}y^{\prime}} {-e^{-t}} 
 =
 -2y^{\prime}e^t \left[
 \frac{y}{x}\frac{d^{2}x}{dy^{2}} +
   \frac{x-\frac{dx}{dy}y}{x^2}\frac{dx}{dy}
 \right]\\&
 \stackrel{by\ (\ref{mauro})}{\sim} 
 -2y^{\prime}e^t
 \left[
    \frac{y}{x}\frac{d^{2}x}{dy^{2}} -\frac{1}{2} \left(\frac{dy}{dx}\right)^2\frac{d^2x}{dy^2} \frac{dx}{dy}
 \right] \rightarrow \sqrt{2}
 \left[ \frac{1}{2} \left.\frac{dy}{dx}\right\vert_{0} \frac{d^2x}{dy^2}\right]
 \end{split}\label{last}
\end{equation}
By (\ref{paralast}) and (\ref{last}) the term in the square bracket on the right hand side of (\ref{basiceq4}) converges to $0$ as required. This completes the proof in Case 3 and the proof of Theorem 1(b). \hfill  \rule{0.5em}{0.5em}
\begin{remark}
 The convexity result posited in Theorem \ref{mainth} also holds for bounded convex domains $\Omega$ with piece wise $C^2$ boundary which consists of either segments or, $C^2$ curves with nno-vanishing curvature. This follows by combining parts (a) and (b) of Theorem \ref{mainth} and the fact that the distance funtion studied in the above proof  has two summands  each of which was treated separetely and shown to be convex.
 \end{remark}
\section{Examples}

We will construct an example which demonstrates the necessity of the curvature condition 
in Theorem \ref{mainth}b. 
Apart from asymptotic geodesics, there are two more cases: intersecting and
disjoint geodesics. In the case $\Omega$ is a convex polytope, we provide
below examples showing that convexity does not hold neither for intersecting
nor for disjoint geodesics.

\textbf{Example 1 (disjoint geodesics):} let $\Omega$ be the interior of the
trapezoid with vertices $\left(  2,0\right)  ,\left(  3,1\right)  ,\left(
-2,1\right)  $ and $\left(  -1,0\right)  .$ Let $f$ be the geodesic whose
image is the intersection of $\Omega$ with the line $x=1$ and $f\left(
0\right)  =\left(  1,\frac{1}{2}\right)  , $ having arc length
parametrization. Similarly, let $g$ be the geodesic with image the
intersection of $\Omega$ with the line $x=0$ and $g\left(  0\right)  =\left(
0,\frac{1}{2}\right)  ,$ see Figure \ref{example1}. By Lemma \ref{xift}, the Euclidean distance
$\left \vert g\left(  t\right)  -g\left(  +\infty \right)  \right \vert $ is
$\frac{1}{e^{t}+1}\equiv E\left(  t\right)  ,$ in other words, the horizontal
lines $y=\frac{1}{e^{t}+1},t\in \left(  0,\infty \right)  $ intersect the images
of the geodesics $f$ and $g$ at the points $f\left(  t\right)  $ and $g\left(
t\right)  $ respectively. The distance function is given by
\[
h\left(  f\left(  t\right)  ,g\left(  t\right)  \right)  =\log \frac{2+E\left(
t\right)  }{1+E\left(  t\right)  }+\log \frac{2+E\left(  t\right)  }{1+E\left(
t\right)  }.
\]
An elementary calculation shows that
\[
h^{\prime \prime}=2\frac{\Phi}{\left[  2+E\left(  t\right)  \right]
^{2}\left[  1+E\left(  t\right)  \right]  ^{2}}%
\]
where the dominant summand of $\Phi$ is
\[
-2E^{\prime \prime}\left(  t\right)  =-2\frac{e^{2t}-e^{t}}{\left(
e^{t}+1\right)  ^{3}}.
\]
Thus, for large enough $t,$ the function $h\left(  f\left(  t\right)
,g\left(  t\right)  \right)  $ is not convex.\\[2mm]\begin{figure}[ptb]
\begin{center}
\includegraphics
[scale=1.2]
{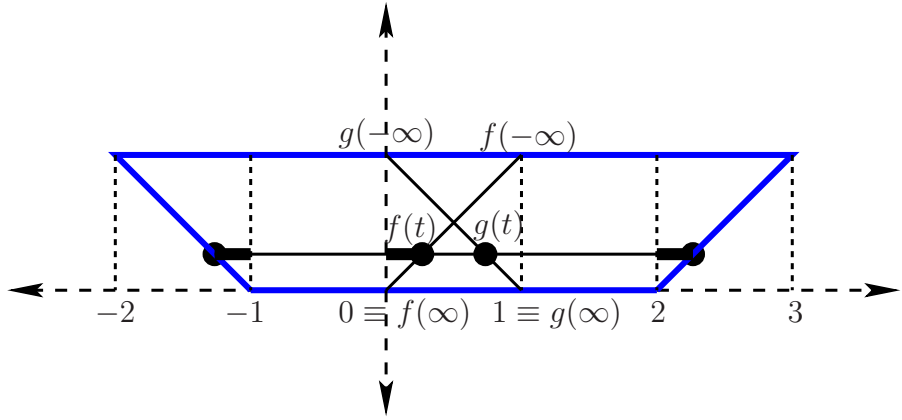}
\end{center}
\par
\begin{picture}(22,12)
\put(215,135){$f(-\infty)$}
\put(162,136){$g(-\infty)$}
\put(213,101){$g(t)$}
\put(179,100){$f(t)$}
\put(119,68){$-1$}
\put(70,68){$-2$}
\put(162,68){$0\equiv f(\infty)$}
\put(220,68){$1\equiv g(\infty)$}
\put(280,68){$2$}
\put(332,68){$3$}
\end{picture}
\caption{The thick segments are equal of Euclidean length $\frac{\sqrt{2}}%
{2}\left \vert f\left(  t\right)  - f\left(  +\infty \right)  \right \vert
=\frac{\sqrt{2}}{2}\frac{\sqrt{2}}{e^{t}+1}=E(t).$}\label{example2}
\end{figure}
\textbf{Example 2 (intersecting geodesics):} let $\Omega$ be as above.
Let $f,g$ be the geodesics with $\operatorname{Im}f=\Omega \cap \left \{
y=x\right \}  $ and $\operatorname{Im}g=\Omega \cap \left \{  y=-x+1\right \}  $
respectively and $f\left(  0\right)  =g\left(  0\right)  =$ $\left(  \frac
{1}{2},\frac{1}{2}\right)  ,$ see Figure \ref{example2}. By the same procedure as in the previous example
we obtain
\[
h\left(  f\left(  t\right)  ,g\left(  t\right)  \right)  =\log \frac
{2}{1+2E\left(  t\right)  }+\log \frac{2}{1+2E\left(  t\right)  }.
\]
An analogous elementary calculations shows that
\[
h^{\prime \prime}=2\frac{\Phi}{\left[  1+2E\left(  t\right)  \right]  ^{2}}%
\]
where the dominant summand of $\Phi$ is $-E^{\prime \prime}\left(  t\right)
\frac{e^{2t}-e^{t}}{\left(  e^{t}+1\right)  ^{3}}.$\\[2mm]
\begin{figure}[ptb]
\begin{center}
\includegraphics
[scale=1.2]
{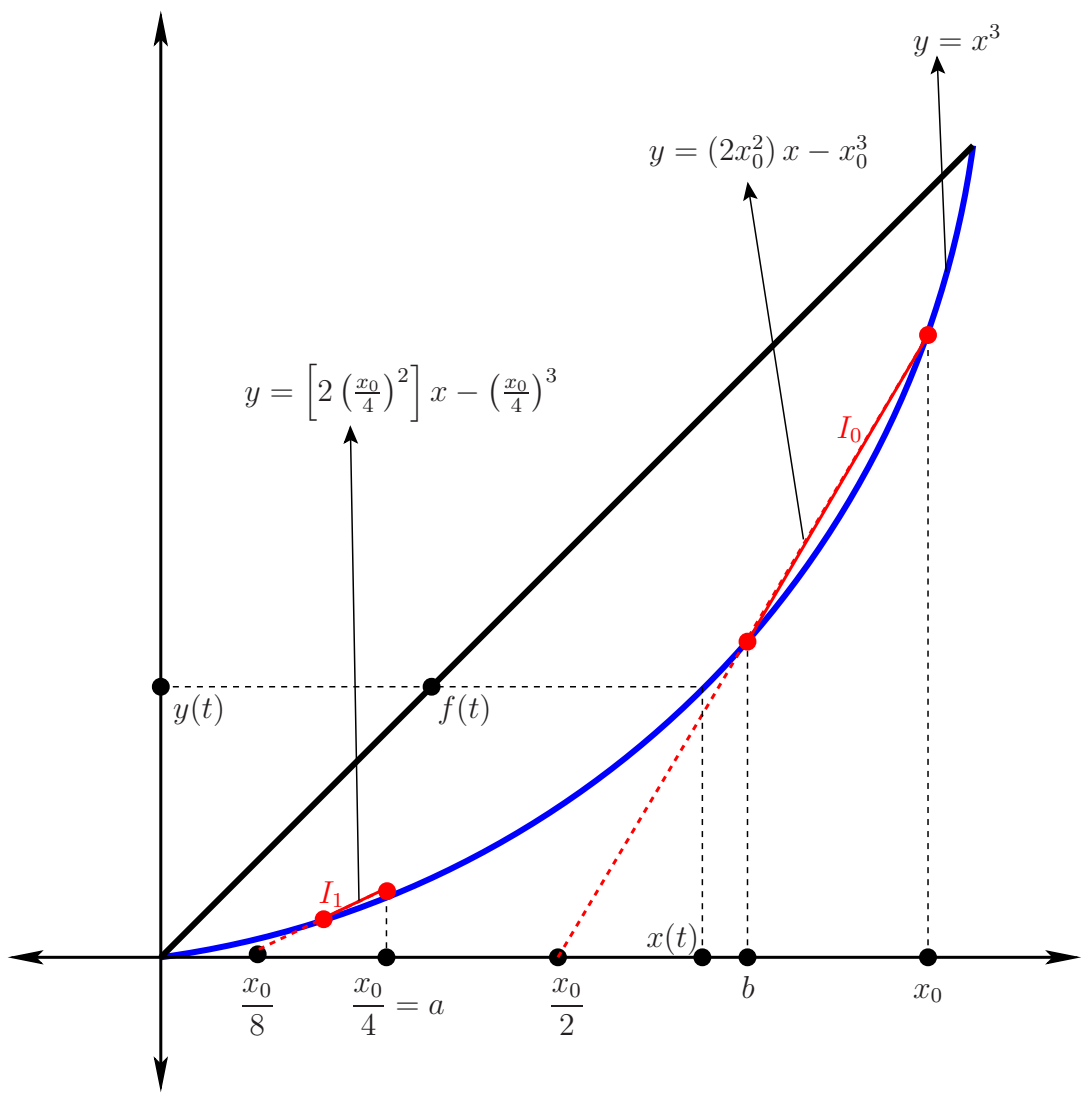}
\end{center}
\par
\begin{picture}(22,12)
\put(165,174){$f(t)$}
\put(65,174){$y(t)$}
\put(244,87){$x(t)$}
\put(132,62){$\displaystyle\frac{x_0}{4}=a$}
\put(90,62){$\displaystyle\frac{x_0}{8}$}
\put(207,62){$\displaystyle\frac{x_0}{2}$}
\put(280,68){$b$}
\put(345,68){$x_0$}
\put(345,427){$y=x^3$}
\put(245,385){$y=\left(2x_0^2 \right) x - x_0^3$}
\put(92,294){$y=\left[2\left(\frac{x_0}{4}\right)^2 \right] x - 
                               \left(\frac{x_0}{4}\right)^3$}
\put(316,280){$\color{red} I_0$}        
\put(120,104){$\color{red} I_1$}                         
\end{picture}
\caption{The segments $I_n$ for $n=0,1 .$}\label{segments01}
\end{figure}

\textbf{Example 3:} We will construct a (convex) $C^2$ curve whose curvature is zero at exactly one point, positive at every other point and two  geodesics asymptotic at the point of zero curvature such that their distance function is not convex.

Consider the convex domain $\Omega$ bounded below by the function
\[ y=\left| x^3 \right| , x\in[-1,1] .\]
Let $f$ (resp. $g$) be the geodesic line whose image is the intersection of $\Omega$ with the line $y=x$ (resp. $y=-x$). We may assume that $\Omega$ is a bounded convex domain containing the above mentioned geodesics. Although Theorem \ref{mainth}b does not apply because 
$\partial \Omega$ has curvature $0$ at the point $(0,0),$ it can be shown that the distance function 
\[ 
D(t):=d\left( f(t),g(t) \right) = 2 \log \frac{x(t)+y(t)}{x(t)-y(t)}
\]
is in fact convex for sufficiently large $t.$ We will alter $\Omega$ by replacing a subarc of its boundary by a segment so that the distance function will no longer be convex at the corresponding time interval. Then we will repeat the same process infinitely many times to ensure that convexity does not hold for large $t$ and we will take appropriate care for the $C^2$ property. \\
By symmetry, we will restrict our attention to $x(t)>0 .$  
\begin{lemma}
 Let $y(t) = \displaystyle \frac{c}{e^t +1}, c>0$ and $y(t) = A x(t) +B$ with 
 $A\in (0,1)$ and $ B<0.$  Then the first and second derivative of  
 $\displaystyle D (t) = 2\log \frac{x(t)+y(t)}{x(t)-y(t)}$ are as follows:
 \[
  D^{\prime} (t) = 4\,\, \frac{Ax-y}{x^2 -y^2}\,\,\frac{1}{A}\,\left( \frac{y^2}{c}-y \right)
 \]
\[
  D^{\prime\prime} (t)\,\, \mathcal{C} = 
    2\left( Ay-x\right) \left( -\frac{y^2}{c} +y \right) + 
                   \left( x^2 -y^2 \right) \left( -\frac{2yA}{c} +A \right)
\]
where $\mathcal{C}$ is the following positive real number 
$\displaystyle \frac{A^2\left( x^2 -y^2 \right)^2}
                    {-4B\left( -\frac{y^2}{c} +y \right)}  .$
\end{lemma}

For $x_0 \in (0,1) $ consider the line determined by the points 
$\left( x_0 , x_0^3 \right)$ and $\left( \frac{x_0}{2} ,0 \right)$ whose equation is 
$y=\left(2x_0^2 \right) x - x_0^3 . $ This line determines a segment $I_0$ with endpoints
$\left(  x_0 , x_0^3\right) $ and $\left(  b,b^3\right)  $ for some $b<x_0$ which can be computed explicitly (see Figure \ref{segments01}). Using the previous Lemma
it is easy to see that for $t_0$ such that $x\left( t_0\right) = x_0$  we have $A=2x_0^2 ,$
$y(t_0)=x_0^3$ and, thus,
\[
  D^{\prime\prime} (t_0)\, \mathcal{C} = -\frac{2}{c}x_0^7 +2x_0^8
\]
which is negative for sufficiently small $x_0 .$
Clearly, if the subarc of $\partial \Omega$ determined by the points 
$\left( b , b^3 \right)$ and $\left( x_0 , x_0^3 \right)$
is replaced by the segment $I_0$ then $D(t)$ will not be convex near $t_0 .$
The same non-convexity property can be obtained by replacing the above 
mentioned subarc of  $\partial \Omega$ by a $C^2$ arc
\[
 \sigma_1 :[b,x_0] \longrightarrow \left[ b^3 ,x_0^3 \right]
\]
of constant and sufficiently small curvature. \\
Using $ x_n = \displaystyle \frac{x_0}{2^{2n}} , n\in \mathbb{N}$ as starting point
we obtain the corresponding intervals $I_n$ with endpoints on $\partial\Omega$ and we perform the same replacement for all $n\in \mathbb{N} $ using $C^2$ arcs $ \sigma_n $
of constant and sufficiently small curvature. Moreover, we may arrange so that the curvature of 
each $ \sigma_n $ $\rightarrow 0$ as $n\rightarrow \infty .$\\
This guarantees that the distance function $D(t)$ with respect to the new (altered) convex domain, denoted again 
by $\Omega ,$ cannot be convex for $t$ large enough.\\
The endpoint $\left(  b,b^3\right)$ of the interval $I_0$ can be computed explicitly 
but we will only need the fact that 
\begin{equation}
  b\in \left( \frac{x_0}{2} , \frac{3x_0}{4}\right) . \label{betamiddle}\end{equation}
Denote by $a$ the point $x_1 = \displaystyle \frac{x_0}{4} .$ \\
Our Final Step is to replace the subarc of $\partial\Omega$ with endpoints  
$\left( a , a^3 \right)$ and $\left( b , b^3 \right)$ by a $C^2$ curve 
\[ \sigma_{1,2} : [a,b]\longrightarrow \left[ a^3 ,b^3 \right] \]
so that  the first and second derivatives of $\sigma_{1,2}$ matches those 
of $\sigma_1$ and $\sigma_2 $ at the appropriate points.
\begin{lemma}\label{lemmaintineq}
 Let $[\alpha,\beta] $ be an interval. Let $\alpha^{(0)} , \alpha^{(1)} , \alpha^{(2)}$ and 
 $\beta^{(0)} , \beta^{(1)} ,\beta^{(2)}$ be positive
 real numbers satisfying $  0<\alpha^{(0)}<\beta^{(0)} ,   0<\alpha^{(1)}<\beta^{(1)} $ and 
 \begin{equation}
    \alpha^{(1)}(\beta-\alpha )< \beta^{(0)} -\alpha^{(0)} < \beta^{(1)} (\beta-\alpha) \label{intineq}
 \end{equation}    
Then there exists a $C^2$ function 
$\sigma : [\alpha,\beta]\longrightarrow\left[ \alpha^{(0)} ,\beta^{(0)} \right]$
satisfying
\begin{itemize}
  \item $\sigma (\alpha)=\alpha^{(0)} ,\sigma^{\prime} (\alpha) =\alpha^{(1)} , 
  \sigma^{\prime\prime} (\alpha)=\alpha^{(2)} $
  \item $\sigma (\beta)=\beta^{(0)} ,\sigma^{\prime} (\beta) =\beta^{(1)} , 
  \sigma^{\prime\prime} (\beta) =\beta^{(2)}$
   \item $\sigma^{\prime\prime} (x)>0$ for all $x\in(\alpha,\beta ).$
\end{itemize}
\end{lemma}
\begin{proof}
 First, we may find a strictly increasing differentiable function 
 \[\sigma^{(1)} : \left[ \alpha , \beta\right] \longrightarrow 
 \left[ \alpha^{(1)} ,\beta^{(1)}\right] \]
 satisfying 
 \[ \lim_{t\rightarrow \alpha} \frac{d}{dt} \sigma^{(1)}(t) = \alpha^{(2)}
 \textrm{\ \ and\ \ }
 \lim_{t\rightarrow \beta} \frac{d}{dt} \sigma^{(1)}(t) = \beta^{(2)}
 \]
 Set $\sigma (t) = \alpha^{(0)} + \int_{\alpha}^{t}\sigma^{(1)}(s) ds $ and we need the following equality to hold 
 \begin{equation}
\int_{\alpha}^{\beta} \sigma^{(1)} (t) dt = \sigma(\beta) - \sigma(\alpha) =
           \beta^{(0)}  - \alpha^{(0)}    .         \label{pat}
 \end{equation}
As \[ \alpha^{(1)}(\beta -\alpha) < \int_{\alpha}^{\beta} \sigma^{(1)} (t) dt <
      \beta^{(1)}(\beta -\alpha)\]
equation (\ref{pat}) can be achieved provided that (\ref{intineq}) holds.   
\end{proof}

In order to use the above Lemma to find $\sigma_{1,2}$ we need to check that the boundary values of $\sigma_1$ and $\sigma_2$ which correspond to the segments $I_0$ and $I_1 $  satisfy the assumptions of the above Lemma. 

The slope of $I_0$ (resp. $I_1$) is $2x_0^2 $ (resp. $2(x_0/4)^2$) so the first two inequalities
required by Lemma \ref{lemmaintineq} clearly hold:
\[
  0<\left(  \frac{x_0}{4} \right)^3<b^3 \textrm{\ and\ } 0< 2 \left(  \frac{x_0}{4} \right)^2 <
  2x_0^2 .
\]
For condition (\ref{intineq}) of Lemma
\ref{lemmaintineq} we need to check that 
\[
 2 \left(  \frac{x_0}{4} \right)^2 \left( b- \frac{x_0}{4} \right) <
b^3- \left(  \frac{x_0}{4} \right)^3 <  2x_0^2 \left( b-  \frac{x_0}{4} \right).
\]
For the right hand side inequality and using (\ref{betamiddle}) it suffices to check that 
\[
\left(  \frac{3x_0}{4} \right)^3 - \left(  \frac{x_0}{4} \right)^3 < 2x_0^2
\left(\frac{x_0}{2} -  \frac{x_0}{4} \right)
\]
which is equivalent to $\frac{26}{64} x_0^3 < \frac{1}{2}x_0^3$ which holds. For the left
hand side inequality and using again (\ref{betamiddle}) it suffices to check that 
\[
2\left(  \frac{x_0}{4} \right)^2\left(\frac{3x_0}{4} -  \frac{x_0}{4} \right)
< \left(  \frac{x_0}{2} \right)^3 - \left(  \frac{x_0}{4} \right)^3 
\]
which is equivalent to $\frac{x_0^3 }{32} < \frac{7}{64}x_0^3$ which holds.

As we have the liberty to choose the arcs $\sigma_n$ arbitrarily close to the segments $I_n ,$ it follows that the inequalities required in Lemma \ref{lemmaintineq} hold for any pair
of segments $I_n , I_{n+1} .$ Therefor,
the above procedure can be followed in an identical way to construct the curves $\sigma_{n,n+1}$ 
joining the curves $\sigma_n$ $\sigma_{n+1}$ for all $n.$ The final curve obtained in this way 
is clearly a $C^2$ curve with zero curvature only at the point $(0,0) $ and, by construction,
the distance function between the geodesics $f,g$ which are asymptotic at $(0,0)$ is not convex.

\end{document}